\documentclass[11 pt]{amsart}

\usepackage[utf8]{inputenc}
\usepackage{amsmath,amsthm,amssymb,amsfonts,amscd,color,xcolor,mathtools}
\usepackage{bbm, bm}
\usepackage[top=3cm, bottom=3cm,inner=3cm, outer=3cm]{geometry}	
\usepackage[british]{babel} 

\usepackage{thmtools} 
\usepackage{esint} 
\usepackage{csquotes} 

\usepackage{enumitem}
\setlist{topsep=4pt,itemsep=1pt,partopsep=1pt,parsep=2pt}

\usepackage{subfiles}

\makeatletter
\def\paragraph{\@startsection{paragraph}{4}%
  \z@\z@{-\fontdimen2\font}%
  {\normalfont\bfseries}}
\makeatother

\newtheorem{theorem}{Theorem}[section]

\newtheorem{lemma}[theorem]{Lemma}
\newtheorem{corollary}[theorem]{Corollary}

\theoremstyle{definition}
\newtheorem{definition}[theorem]{Definition}

\newtheorem{remark}[theorem]{Remark}

\numberwithin{equation}{section}

\usepackage{hyperref}
\usepackage[capitalise]{cleveref}
\crefname{equation}{}{}
\creflabelformat{equation}{\textup{(#2#1#3)}}

\newcommand*\diff{\mathop{}\!\mathrm{d}}
\newcommand{\Mod}[1]{\, (\mathrm{mod} \, #1)}

\let\a\relax
\newcommand{\a}{\mathbf{a}}

\let\b\relax
\newcommand{\BBox}{\mathcal{B}}

\newcommand{\C}{\mathbb{C}}
\newcommand{\Constant}{\mathcal{C}}

\newcommand{\h}{\mathbf{h}}

\newcommand{\Interval}{\mathcal{I}}

\newcommand{\Jintegral}{\mathfrak{J}}

\newcommand{\mi}{\mathfrak{m}}
\newcommand{\Ma}{\mathfrak{M}}

\newcommand{\N}{\mathbb{N}}
\newcommand{\nn}{\mathfrak{n}}
\newcommand{\Na}{\mathfrak{N}}

\newcommand{\Pa}{\mathfrak{P}}

\newcommand{\Q}{\mathbb{Q}}

\newcommand{\R}{\mathbb{R}}

\newcommand{\Series}{\mathfrak{S}}

\newcommand{\T}{\mathbb{T}}

\newcommand{\x}{\mathbf{x}}

\newcommand{\Z}{\mathbb{Z}}
\newcommand{\Zero}{\mathcal{Z}}

\let\eps\relax
\newcommand{\eps}{\varepsilon}

\newcommand{\norm}[1]{\left\| #1 \right\|}

\newcommand{\set}[1]{\left\{ #1 \right\}}
\newcommand{\of}[1]{\left( #1 \right)}

\newcommand{\abs}[1]{\left| #1 \right|}

\usepackage[urldate=long, backend=biber, style = alphabetic]{biblatex}
\title[The Hasse principle for diagonal forms on adjacent degree hypersurfaces]{The Hasse principle for diagonal forms restricted to a hypersurface of adjacent degree}
\author{Anna Theorin Johansson}
\address{Department of Mathematical Sciences, Chalmers University of Technology and the University of Gothenburg, SE-412 96 Gothenburg, Sweden}
\email{annatheo@chalmers.se} 
\date{}

\begin{document}

\begin{abstract}
    We investigate the Hasse principle for Diophantine systems consisting of one diagonal form of degree $k$ and one general form of degree $k-1$. By refining the method of Brandes and Parsell \cite{brandes_hasse_2021} in this specific setting, we improve the bound $n > 2^k k$ to $n > 2^{k-1}(2k-1)$; in particular, the requirement $n > 24$ in the case of degrees three and two is relaxed to $n > 20$. 
\end{abstract}

\maketitle

\section{Introduction}
Suppose that we are given a pair of forms $F, G \in \Z[x_1, \dots, x_n]$ of degree $k$ and $d$ respectively. The goal of this paper is to give an asymptotic formula for the number $N_{F, G}(X)$ of integer points $\x \in [-X, X]^n$ satisfying $F(\x) = G(\x) = 0$. For $n$ sufficiently large and for suitably non-singular forms, we expect this to be of the form
\begin{equation}\label{eq:asymp}
    N_{F, G}(X) = X^{n-k-d} \of{\Constant + O\of{X^{-\nu}}},
\end{equation}
where $\Constant \geq 0$ is a constant and $\nu> 0$ is a small number, both dependent at most on the forms $F$ and $G$. In particular, the objective is to find the smallest integer $n_0 = n_0(k, d)$ such that \cref{eq:asymp} is satisfied for all $n > n_0$. 

If both $F$ and $G$ are diagonal, Wooley's efficient congruencing method gives
\begin{equation*}
    n_0(k, d) \leq k(k+1)
\end{equation*}
for any two distinct degrees $k > d$, with better bounds if $d$ is much smaller than $k$ (see \cite[\S 14]{wooley_nested_2019}). In the case $k = 3$ and $d = 2$, one can achieve 
\begin{equation*}
    n_0(3, 2) \leq \frac{32}{3},
\end{equation*}
so that $n = 11$ would suffice \cite{wooley_rational_2015}, which reaches the square root barrier.

On the other hand, without any assumption of diagonality, one must resort to the work of Browning and Heath-Brown in \cite{browning_forms_2017}, which studies systems of forms of differing degrees in full generality. These methods give an upper bound for $n_0$ which grows exponentially in both $k$ and $d$; to be precise, in \cite[Corollary 1.8]{browning_forms_2017} it is shown that
\begin{equation}\label{eq:n_0_BHB}
    n_0(k, d) \leq (2+d)(k-1)2^{k-1} + d2^{d-1}. 
\end{equation}
With $k = 3$ and $d = 2$, this corresponds to
\begin{equation*}
    n_0(3, 2) \leq 36,
\end{equation*}
but this particular result can be sharpened. Indeed, Browning, Dietmann and Heath-Brown, in \cite[Theorem 1.3]{browning_rational_2015}, show that for one cubic and one quadratic, one can actually take
\begin{equation}\label{eq:n_0_32_BDHB}
    n_0(3, 2) \leq 28.
\end{equation}

In the case where $F$ is diagonal but no particular shape restriction is put on $G$, Brandes and Parsell showed in \cite{brandes_hasse_2021} that one can use the diagonal structure of $F$ to replace the exponential growth in $k$ in \cref{eq:n_0_BHB} by a quadratic one. With $F$ diagonal of degree $k$ and $G$ of degree $d < k$, \cite[Theorem 1.1]{brandes_hasse_2021} gives
\begin{equation}\label{eq:n_0_kd_BP}
    n_0(k, d) \leq 
    \begin{cases}
        2^k(d+1) & \quad \text{if } d+1 \leq k \leq d+4, \\
        2^d(26 + 32d) & \quad \text{if } k = d + 5, \\
        2^d\left[(2d+1)k^2 - L_d(k) \right] & \quad \text{if } k \geq d + 6,
    \end{cases} 
\end{equation}
with $L_d(k)$ an explicit function of $d$ and $k$; see \cite{brandes_hasse_2021} and the corresponding forthcoming corrigendum for its precise form. This is an improvement on \cref{eq:n_0_BHB} in all cases; in particular, if $d = k-1$, this is
\begin{equation}\label{eq:n_0_k_BP}
    n_0(k, k-1) \leq 2^k k.
\end{equation}
For instance,
\begin{equation}\label{eq:n_0_32_BP}
    n_0(3, 2) \leq 24
\end{equation}
is a clear improvement on \cref{eq:n_0_32_BDHB}.

The goal of this paper is to streamline the use of the Weyl inequality in the proof of the above in the case $d = k-1$ to improve \cref{eq:n_0_k_BP}. In order to state our result, we introduce the local solution densities associated with the system $F(\x) = G(\x) = 0$. For $q \in \N$, let
\begin{equation*}
    \Gamma(q) := \# \set{\x \in (\Z/q\Z)^n: F(\x) \equiv G(\x) \equiv 0 \mod q},
\end{equation*}
and for each prime $p$, set 
\begin{equation*}
    \chi_p := \lim_{h \to \infty} \of{p^{h}}^{2-n} \Gamma(p^h).
\end{equation*}
Next, define for $\bm{\gamma} = (\gamma_k, \gamma_d) \in \R^2$ the integral
\begin{equation}\label{eq:J_def}
    J(\bm{\gamma}) := \int_{[-1, 1]^n} e \of{\gamma_k F(\bm{\xi}) + \gamma_d G(\bm{\xi})} \diff \bm{\xi},
\end{equation}
and set
\begin{equation*}
    \chi_{\infty} := \lim_{R\to \infty}\int_{[-R, R]^2} J(\bm{\gamma}) \diff \bm{\gamma}.
\end{equation*}
Let 
\begin{equation*}
    \Constant_{F, G} := \chi_{\infty} \prod_p \chi_p.
\end{equation*}
We will see in \cref{sec:major} that the limits defining $\Constant_{F, G}$ do indeed exist. Heuristically, $\chi_p$ can be interpreted as the density of solutions of $F(\x) = G(\x) = 0$ in $\Z_p^n$, and $\chi_\infty$ in turn as the density of real solutions in the box $[-1, 1]^n$. 

The following is our main result.
\begin{theorem}\label{thm:main}
    Let $k \geq 3$ and $d = k-1$. Let $F, G \in \Z[x_1, \dots, x_n]$ be a pair of non-singular forms of degree $k$ and $d$ respectively, where $F$ is diagonal, i.e. of the shape 
    \begin{equation*}
       F(\x) = F(x_1, \dots, x_n) = c_1 x_1^k + \dots + c_n x_n^k.
    \end{equation*}
    Suppose that
    \begin{equation}\label{eq:n20}
        n > 2^{k-1} (2k-1).
    \end{equation}
    Then there is a $\nu > 0$ such that
    \begin{equation*}
        N_{F, G}(X) = X^{n-k-d} \of{\Constant_{F, G} + O\of{X^{-\nu}}}.
    \end{equation*}
    Moreover, $\Constant_{F, G} > 0$ whenever the system $F(\x) = G(\x) = 0$ has a non-singular solution in $\R^n$ and in $\Q_p^n$ for every prime $p$.
\end{theorem}
In particular, if the system $F(\x) = G(\x) = 0$ has non-singular solutions over $\R$ and over every $\Q_p$, then $N_{F, G}(X) \gg X^{n-k-d}$, implying that the system has non-trivial integer solutions; that is, the smooth Hasse principle holds for such systems.

For the case of $k = 3$ and $d = 2$, this improves \cref{eq:n_0_32_BP} by four variables:
\begin{corollary}
    Let $C, Q \in \Z[x_1, \dots, x_n]$ be a pair of non-singular forms, where $C$ is cubic and diagonal of the shape 
    \begin{equation*}
       C(\x) = C(x_1, \dots, x_n) = c_1 x_1^3 + \dots + c_n x_n^3,
    \end{equation*}
    and $Q$ is quadratic. Suppose that
    \begin{equation*}
        n > 20.
    \end{equation*}
    Then there is a $\nu > 0$ such that
    \begin{equation*}
        N_{C, Q}(X) = X^{n-5} \of{\Constant_{C, Q} + O\of{X^{-\nu}}},
    \end{equation*}
    where $\Constant_{C, Q} \geq 0$.
\end{corollary}
While our methods, combined with those of \cite[\S 2]{brandes_hasse_2021}, would work also for $d < k-1$, this would give
\begin{equation*}
    n > 2^{k-1}(k+d),
\end{equation*}
which is weaker than \cref{eq:n_0_kd_BP} for all $d \leq k-3$ (and equal for $d = k-2$).

\paragraph{Outline of paper}
This paper revisits a special case of \cite{brandes_hasse_2021} and obtains an improvement using only standard circle method results. By avoiding technical complications, the argument becomes largely self-contained, which may be helpful for readers looking for an accessible entry point or a basis for future work.

The proof of \cref{thm:main} uses the Hardy--Littlewood circle method and follows the argument of \cite{brandes_hasse_2021}. In \cref{sec:weyl}, we define our one-dimensional major arcs $\Ma(\theta) \subseteq \T$ consisting of those $\alpha_k$ that lie close to a rational number with small denominator, and collect the necessary Weyl-type results. In \cref{sec:minor}, we express the counting function $N_{F, G}(X)$ as an integral of an exponential sum $S(\alpha_k, \alpha_d)$, and use Weyl differencing to show that the contribution from $\alpha_k$ not in $\Ma(\theta)$ is negligible. For $\alpha_k \in \Ma(\theta)$, an argument of Browning and Heath-Brown \cite{browning_forms_2017} shows that $S(\alpha_k, \alpha_d)$ is small unless $\alpha_d$ is also well approximated by a rational number whose denominator is related to that of $\alpha_k$, which leads us to a set of two-dimensional major arcs $\Na(\eta) \subseteq \T^2$. Once we have shown that the contribution from its complement is negligible as well, we show in \cref{sec:major} with standard arguments that evaluating the integral over $\Na(\eta)$ yields precisely the main term in the statement of the theorem.

\paragraph{Notation}
We will use the following conventions throughout:
\begin{itemize}
    \item For a vector $\x \in \R^m$, define $\abs{\x} := \norm{\x}_{\infty} = \max_{i \in \set{1, \dots, m}} \abs{x_i}$. 
    \item For $\alpha \in \R$, let $e(\alpha) = e^{2 \pi i \alpha}$ and $e_q(\alpha) = e(\alpha/q)$ whenever $q \in \N$. 
    \item Let $\T = \R/\Z$, and write $\operatorname{vol}$ for the Lebesgue measure on $\T$ and on $\T^2$. 
    \item For $\beta\in \R$, write $\norm{\beta} := \min_{b \in \Z} \abs{\beta - b}$ for the distance from $\beta$ to the nearest integer. Since $\norm{\beta}$ depends only on the residue class of $\beta$ modulo $1$, this also defines $\norm{\alpha}$ for $\alpha \in \T$, which allows us to use $\norm{\cdot}$ on both $\R$ and $\T$.
    \item If $f: \T^m\to \C$ is integrable, define
    \begin{equation*}
        \oint f(\bm{\xi}) \diff \bm{\xi} := \int_{\T^m} f(\bm{\xi}) \diff \bm{\xi},
    \end{equation*}
    where the dimension $m$ is always clear from the context.
    \item When $\eps$ appears in a statement, it is asserted that 
    the statement holds for every $\eps > 0$, with implied constants that may depend on $\eps$. 
    \item Whenever $X$ appears in a statement, it is asserted that the statement holds for all sufficiently large $X$, that is, for all $X \geq X_0$, where $X_0$ may depend on every parameter other than $X$ (including $\eps$). 
    \item The implied constants in the Landau and Vinogradov notations are always allowed to depend on every parameter but $X$, unless otherwise specified. 
\end{itemize}

\section{Weyl estimates}\label{sec:weyl}
For $q \in \N$, $a \in \Z$ and $\theta \in (0, 1]$, let
\begin{equation*}
    \Ma(q, a; \theta) := \set{ \alpha \in \T: \abs{\alpha - \frac{a}{q}} \leq q^{-1} X^{-k + \theta} }.
\end{equation*}
\begin{definition}\label{def:major_arcs}
    For $\theta \in (0, 1]$, let $\Ma(\theta)$ be the union of the arcs $\Ma(q, a; \theta)$ with $(a, q) = 1$ and $1 \leq a \leq q \leq X^{\theta}$. Set $\mi(\theta) = \T \setminus \Ma(\theta)$.
\end{definition}
The arcs comprising $\Ma(\theta)$ are pairwise disjoint for all $\theta \in (0, 1]$, and
\begin{equation}\label{eq:major_volume}
    \operatorname{vol} \Ma(\theta) \ll \sum_{q \leq X^{\theta}} \sum_{\substack{a = 1 \\ (a, q) = 1}}^q \frac{X^{-k + \theta}}{q} \ll X^{-k + 2 \theta}. 
\end{equation}

Let $\ell \in \Z[x]$ be a linear polynomial and let $\Interval \subseteq [-X, X]$ be an interval. Set
\begin{equation*}
    \varphi(\alpha) = \varphi(\alpha; \Interval) := \sum_{x \in \Interval \cap \Z} e \of{\alpha \ell(x)}.
\end{equation*}
We want to show that sums of $\varphi(\alpha)$ are small when $\alpha$ lies on the minor arcs.
\begin{lemma}\label{lemma:weyl1}
    Suppose that $\alpha \in \mi(\theta)$ for some $\theta \in (0, 1]$. Then 
    \begin{equation*}
        \sum_{h = 1}^{X^{k-1}} \abs{\varphi(h \alpha)} \ll X^{k - \theta + \eps}.
    \end{equation*}
\end{lemma}
\begin{proof}
    This is shown using Dirichlet's theorem and \cite[Lemma 2.2]{vaughan_hardy-littlewood_1997}.
\end{proof}

Together with a standard divisor bound, \cref{lemma:weyl1} proves the following result.
\begin{lemma}\label{lemma:weyl2}
    Suppose that $\alpha \in \mi(\theta)$ for some $\theta \in (0, 1]$. Then
    \begin{equation*}
        \sum_{\substack{\abs{h_1}, \dots, \abs{h_{k-1}} \leq X \\ h_1 \cdots h_{k-1} \neq 0}} \abs{\varphi (h_1 \cdots h_{k-1} \alpha)} \ll X^{k - \theta + \eps}.
    \end{equation*}
\end{lemma}
\begin{remark}\label{rmk:constant}
    Note that the implied constant depends only on the leading coefficient of the linear polynomial $\ell$, and in particular not on the constant term in $\ell$ or the interval $\Interval$.
\end{remark}

\section{Minor arc contribution}\label{sec:minor}
The goal is to obtain an asymptotic formula for the quantity
\begin{equation*}
    N_{F, G}(X) = \sum_{\substack{\x \in  [-X, X]^n \cap \Z^n \\
    F(\x) = G(\x) = 0}} 1.
\end{equation*}
Using the identity
\begin{equation*}
    \oint e(\alpha m) \diff \alpha = 
    \begin{cases}
        1 & \quad \text{if } m = 0, \\
        0 & \quad \text{if } m \in \Z \setminus\set{0},
    \end{cases}
\end{equation*}
we can reformulate the sum as
\begin{equation*}
     N_{F, G}(X) = \oint S(\bm{\alpha}) \diff \bm{\alpha}.
\end{equation*}
Here, for $\bm{\alpha} = (\alpha_k, \alpha_d) \in \T^2$, we define
\begin{equation*}
    S(\bm{\alpha}) = S(\alpha_k, \alpha_d; X) := \sum_{\x \in [-X, X]^n \cap \Z^n} e \of{\alpha_k F(\x) + \alpha_d G(\x)}.
\end{equation*}

We begin by using \cref{lemma:weyl2} to bound $S(\bm{\alpha})$ on the minor arcs.
\begin{lemma}\label{lemma:ch}
    For every $\theta \in (0, 1]$ and every $\alpha_d \in \T$, we have
    \begin{equation*}
        \sup_{\alpha_k \in \mi(\theta)} \abs{S(\alpha_k, \alpha_d)} \ll X^{n - \frac{1}{2^d} n \theta + \eps}.
    \end{equation*}
\end{lemma}
\begin{proof} 
    This goes along the lines of \cite[\S 5]{brandes_hasse_2021}. 
    
    For $m \geq 1$, $\h \in \Z^m$ and a function $f$ on $\Z^m$, let $\Delta_{\h} f(\x) := f(\x + \h) - f(\x)$ denote the forward difference operator. For $\h = (h_1, \dots, h_d) \in \Z^d$, set
    \begin{equation}\label{eq:p_shape}
        p_{\h}(x) := \Delta_{h_1} \cdots \Delta_{h_d} x^k = h_1 \cdots h_d \ell_{\h}(x),
    \end{equation}
    where, since $d = k-1$, $\ell_{\h}(x)$ is a linear polynomial whose leading coefficient is independent of $\h$. 
    
    For $\h_1, \dots, \h_d \in \Z^n$, set $\h^{(i)}:= (h_{1, i}, \dots, h_{d, i}) \in \Z^d$ for each $i = 1, \dots, n$. Then 
    \begin{equation*}
        \Delta_{\h_1} \cdots \Delta_{\h_d} \of{\alpha_k F(\x) + \alpha_d G(\x)} = \alpha_k \sum_{i = 1}^n c_i p_{\h^{(i)}}(x_i) + \alpha_d \Delta_{\h_1} \cdots \Delta_{\h_d} G(\x),
    \end{equation*}
    where the last term is a constant independent of $\x$. Applying Weyl differencing $d$ times as in the proof of \cite[Lemma 2.3]{vaughan_hardy-littlewood_1997}, we find sets $I(\h^{(i)}) \subseteq [-X, X] \cap \Z$ such that
    \begin{equation*}
        \abs{S(\bm{\alpha})}^{2^d} \ll X^{\of{2^{d}-d-1}n} \sum_{\abs{\h_1} \leq 2X} \cdots \sum_{\abs{\h_d} \leq 2X} \abs{\sum_{x_1 \in I(\h^{(1)})} \cdots \sum_{x_n \in I(\h^{(n)})} e \of{\alpha_k \of{ \sum_{i = 1}^n c_i p_{\h^{(i)}}(x_i)} }}.
    \end{equation*}

    For $\alpha \in \R$, $\h \in \Z^d$ and $\BBox \subseteq \Z^d$, define
    \begin{align*}
        f(\alpha; \h) &:= \sum_{x \in I(\h)} e \of{\alpha p_{\h}(x)}, \\
        g(\alpha; \BBox) &:= \sum_{\h \in \BBox} \abs{f(\alpha; \h)}, \\ 
        g(\alpha) &:= g\of{\alpha; [-2X, 2X]^d \cap \Z^d}.
    \end{align*}
    Summing over $\h_1, \dots, \h_d \in \Z^n$ is equivalent to summing over $\h^{(1)}, \dots, \h^{(n)} \in \Z^d$, so the right-hand side above factorises. With the general inequality
    \begin{equation*}
        \abs{a_1 \cdots a_m} \leq \abs{a_1}^m + \dots + \abs{a_m}^m \quad \forall a_1, \dots, a_m \in \C,
    \end{equation*}
    applied to the product of the $g(\alpha_k c_i)$, one obtains
    \begin{equation}\label{eq:S_4_intermediate}
        \abs{S(\alpha_k, \alpha_d)}^{2^d} \ll X^{(2^d-d-1)n} \prod_{i = 1}^n g \of{\alpha_k c_i} \ll X^{(2^d-d-1)n} \sum_{i = 1}^n g \of{\alpha_k c_i}^n.
    \end{equation}
    
    The shape of the terms in \cref{eq:p_shape} means that the sum defining $f$ is trivial if $h_1 \cdots h_d= 0$, which suggests that we should treat any terms with $h_i = 0$ for some $i$ separately in the exponential sums above in order to get a decent bound. To this end, the box is split up into a sum over the terms with $h_1 \cdots h_d= 0$ and the complement of this. Set
    \begin{equation*}
        \Zero := \set{\h \in [-2X, 2X]^d \cap \Z^d: h_1 \cdots h_d = 0} \quad \text{and} \quad \Zero^c := ([-2X, 2X]^d \cap \Z^d) \setminus \Zero.
    \end{equation*}
    Clearly
    \begin{equation*}
        \abs{\Zero } \ll X^{d-1},
    \end{equation*}
    since at least one of the components must be zero, and for $\h \in \Zero$, we have
    \begin{equation*}
        p_{\h}(x) =h_1 \cdots h_d \ell_{\h}(x) = 0,
    \end{equation*}
    which means that, for such $\h$, the sum defining $f$ is
    \begin{equation*}
        f(\alpha; \h) = \sum_{x \in I(\h)} 1 \ll X.
    \end{equation*}
    Hence
    \begin{equation*}
        g\of{\alpha_k c_i} \ll X^{d-1} \max_{\h \in \Zero} \abs{f\of{\alpha_k c_i; \h}} + g\of{\alpha_k c_i; \Zero^c} \ll X^d + g\of{\alpha_k c_i; \Zero^c}. 
    \end{equation*}
    
    The bound in \cref{eq:S_4_intermediate} then becomes
    \begin{align*}
        \abs{S(\alpha_k, \alpha_d)}^{2^d} & \ll X^{(2^d - d- 1)n} \sum_{i = 1}^n \of{X^d + g\of{\alpha_k c_i; \Zero^c} }^n \\
        & \ll X^{(2^d -1)n} + X^{(2^d - d- 1)n} \sum_{i = 1}^n g\of{\alpha_k c_i; \Zero^c}^n \\
        & \ll X^{(2^d -1)n} + X^{(2^d - d- 1)n} \max_{i} g\of{\alpha_k c_i; \Zero^c}^{n}.
    \end{align*}
    Taking into account \cref{rmk:constant}, \cref{lemma:weyl2} gives
    \begin{equation*}
        \sup_{\alpha_k \in \mi(\theta)} \sum_{\h \in \Zero^c} \abs{f\of{\alpha_k c_i;\h}} \ll X^{k - \theta + \eps} = X^{d + 1 - \theta + \eps},
    \end{equation*}
    for each $1 \leq i \leq n$, which in particular means that
    \begin{equation*}
        \sup_{\alpha_k \in \mi(\theta)} \abs{S(\alpha_k, \alpha_d)}^{2^d} \ll X^{(2^d -1)n} + X^{(2^d - d- 1)n} \of{X^{d+1 - \theta + \eps}}^{n} \ll X^{(2^d - \theta) n + n \eps}.
    \end{equation*}
    Taking $2^d$th roots, we arrive at
    \begin{equation*}
        \sup_{\alpha_k \in \mi(\theta)} \abs{S(\alpha_k, \alpha_d)} \ll X^{n- \frac{1}{2^d} n \theta  + \frac{n}{2^d} \eps},
    \end{equation*}
    and since $\eps > 0$ is arbitrary, the lemma follows upon replacing $\eps$ by $2^d \eps /n$.
\end{proof}

\begin{remark}\label{rmk:theta1_bound}
    In particular, with $\theta = 1$, this means that for any $\alpha_k \in \mi(1)$ and any $\alpha_d \in \T$, we have
    \begin{equation*}
        \abs{S(\alpha_k, \alpha_d) } \ll X^{\of{1-\frac{1}{2^d}}n + \eps},
    \end{equation*}
    which implies
    \begin{equation*}
        \oint \int_{\mi(1)} \abs{S(\alpha_k, \alpha_d)} \diff \alpha_k \diff \alpha_d \ll X^{\of{1-\frac{1}{2^d}}n  + \eps}. 
    \end{equation*}
\end{remark}

\begin{lemma}\label{lemma:minor_mi}
    Assume that \cref{eq:n20} holds. For $\theta_* \in (0, 1]$, if
    \begin{equation}\label{eq:theta_star}
        \theta_* > \frac{2^d d}{n - 2^{d+1}},
    \end{equation}
    then there is a $\nu > 0$ such that for all $\theta \in [\theta_*, 1]$, we have
    \begin{equation*}
        \oint \int_{\mi(\theta)}\abs{S(\alpha_k, \alpha_d)} \diff \alpha_k \diff \alpha_d \ll X^{n-k-d - \nu}.
    \end{equation*}
\end{lemma}
\begin{proof}
    To prove this, as in the proof of \cite[Lemma 3.2]{brandes_hasse_2021}, we use \cref{lemma:ch} in combination with \cref{eq:n20} and the major arcs volume estimate in \cref{eq:major_volume}.

    Construct a decreasing sequence $(\theta_i)$ of the form $1 = \theta_0 > \theta_1 > \dots > \theta_N = \theta_*$. With the requirement for $\theta_*$ in \cref{eq:theta_star}, it is clear that this can be chosen with $N = O(1)$ and 
    \begin{equation}\label{eq:step_size_theta}
        2^{d+1} \of{\theta_{i-1} - \theta_i} < \of{n-2^{d+1}} \theta_* - 2^d d.
    \end{equation}
    Note that
    \begin{equation*}
        \mi(\theta_*) = \mi(\theta_N) = \mi(\theta_0) \cup \bigcup_{i = 1}^N \of{\mi(\theta_i) \setminus \mi( \theta_{i-1})},
    \end{equation*}
    so it remains to show that $\mi(\theta_0)$ and then each set $\mi(\theta_i) \setminus \mi(\theta_{i-1})$ give an acceptable contribution to the integral.

    For $\theta_0 = 1$, combining \cref{rmk:theta1_bound} with the restriction for $n$ in \cref{eq:n20} yields
    \begin{equation*}
        \oint \int_{\mi(\theta_0)} \abs{S(\alpha_k, \alpha_d)} \diff \alpha_k \diff \alpha_d \ll X^{n-k-d- \nu}
    \end{equation*}
    for some suitable $\nu > 0$. 

    Next, for each $i > 0$, the bound for $S(\bm{\alpha})$ on the minor arcs in \cref{lemma:ch} together with the volume estimate \cref{eq:major_volume} on the major arcs gives
    \begin{align*}
        \oint \int_{\mi(\theta_i) \setminus \mi(\theta_{i-1})} \abs{S(\alpha_k, \alpha_d)} \diff \alpha_k \diff \alpha_d &\ll \operatorname{vol} \Ma(\theta_{i-1}) \sup_{\substack{\alpha_k \in \mi(\theta_i) \\ \alpha_d \in \T}} \abs{S(\alpha_k, \alpha_d)} \\
        &\ll X^{-k + 2 \theta_{i-1}} X^{n - \frac{1}{2^d} n \theta_i + \eps}.
    \end{align*}
    Now \cref{eq:step_size_theta} shows that
    \begin{equation*}
        2 \theta_{i-1} - \frac{n}{2^d}\theta_i = 2\of{\theta_{i-1} - \theta_i} - \of{\frac{n}{2^d}-2}\theta_i < - d,
    \end{equation*}
    which means that each set $\mi(\theta_i) \setminus \mi(\theta_{i-1})$ contributes an acceptable amount. 
\end{proof}

\subsection{Two-dimensional split}\label{sec:minor_split}
We have now shown that unless $\alpha_k$ is well approximated by rationals, the contribution to the minor arc integral is small. What can then be said in the other case? As in \cite{browning_forms_2017}, a similar Weyl argument shows that whenever $\alpha_k$ has a good rational approximation $a/q$, unless $\alpha_d$ is well approximated with a denominator of the form $qr$, the corresponding contribution will also be small. The following lemma makes this precise.

\begin{lemma}\label{lemma:subdivision}
    Let $\kappa > 0$ and $\theta \in (0, 1)$. If $\alpha_k \in \Ma(\theta)$, with corresponding rational approximation $a/q$, then for any $\eta$ in the range
    \begin{equation}\label{eq:range_eta}
        0 < \eta \leq 1 - \theta,
    \end{equation}
    one of the following three statements is true:
    \begin{enumerate}[label={(\Alph*)}]
        \item Either \label{opt:1}
        \begin{equation*}
            \abs{S(\bm{\alpha})} \ll X^{n-\kappa \eta + \eps},
        \end{equation*}
        \item or there is a natural number $r \leq X^{(d-1)\eta}$ such that \label{opt:2}
        \begin{equation*}
            \norm{q r \alpha_d} \ll X^{-d + (d-1)\eta + \theta},
        \end{equation*}
        \item or \label{opt:3}
        \begin{equation*}
            n \leq 2^{d-1} \kappa.
        \end{equation*}
    \end{enumerate}
\end{lemma}

\begin{proof}
    This is \cite[Lemma 6.1]{browning_forms_2017}.
\end{proof}

In order to exclude \labelcref{opt:3}, we assume from now on that
\begin{equation*}
    n > 2^{d-1} \kappa.
\end{equation*}

\begin{remark}
    We henceforth choose $\eta$ and $\theta$ to be related by
    \begin{equation}\label{eq:kappa_eta_theta}
        \kappa \eta = \frac{1}{2^d} n \theta.
    \end{equation}
    To see why this makes sense, combine the bound on $S(\bm{\alpha})$ in \labelcref{opt:1} with that of \cref{lemma:ch}. 
\end{remark}

Motivated by \cref{lemma:subdivision}, we now define a set $\Na(\eta)$ of two-dimensional major arcs, consisting of those $\bm{\alpha}$ for which both $\alpha_k$ and $\alpha_d$ are well approximated in the sense of \labelcref{opt:2}, with complement $\nn(\eta)$. With a suitable choice of the constant $c$ below, every $\bm{\alpha} \in \nn(\eta)$ with $\alpha_k \in \Ma(\theta)$ satisfies \labelcref{opt:1}.

\begin{definition}\label{def:najor_arcs}
    For a constant $c >0$, let $\Na(\eta) = \Na(\eta, \theta)$ be the set of $\bm{\alpha} \in \T^2$ with 
    \begin{align*}
        \norm{\alpha_k q} \leq c X^{-k + \theta}, && \norm{\alpha_d qr } \leq c X^{-d + (d-1)\eta + \theta}
    \end{align*}
    for some $1 \leq q \leq c X^{\theta}$ and $1 \leq r \leq c X^{(d-1)\eta}$. Define $\nn(\eta) = \T^2 \setminus \Na(\eta)$.
\end{definition}

Two remarks are in order:
\begin{enumerate}
    \item Combining the size restrictions on $S(\bm{\alpha})$ in \cref{lemma:ch} and in \labelcref{opt:1} using the equality \cref{eq:kappa_eta_theta} gives
    \begin{equation}\label{eq:S_bound_nn}
        \abs{S(\bm{\alpha})} \ll X^{n - \kappa \eta + \eps}
    \end{equation}
    for all $\bm{\alpha} \in \nn(\eta)$. 
    \item Using the equality \cref{eq:kappa_eta_theta}, we can estimate:
    \begin{equation}\label{eq:volume_bound_Na}
        \begin{split}
            \operatorname{vol} \Na(\eta) &\ll \sum_{1 \leq q \leq cX^\theta} \sum_{1 \leq a \leq q} \frac{X^{-k + \theta}}{q} \sum_{1 \leq r \leq c X^{(d-1) \eta}} \sum_{1 \leq b \leq qr} \frac{X^{-d + (d-1)\eta + \theta}}{qr} \\
            & \ll X^{-k - d + 2(d-1) \eta + 3 \theta} \\
            & \ll X^{-k-d + \of{2(d-1) + \frac{3\cdot 2^d \cdot \kappa }{n}}\eta}.
        \end{split}
    \end{equation}
\end{enumerate}

\subsection{Two-dimensional minor arcs}\label{sec:minor_minor}
We can apply \cref{lemma:minor_mi} and \cref{lemma:subdivision} using the two estimates in \cref{eq:S_bound_nn} and \cref{eq:volume_bound_Na} to bound the integral over $\nn(\eta)$. 
\begin{lemma}\label{lemma:minor_nn}
    Suppose that \cref{eq:n20} holds, together with
    \begin{equation}\label{eq:kappa_n_1}
        \frac{d}{\kappa} + \frac{2^d(d+2)}{n} < 1
    \end{equation}
    and
    \begin{equation}\label{eq:kappa_n_2}
        \frac{2(d-1)}{\kappa} + \frac{3 \cdot 2^d}{n} < 1.
    \end{equation}
    Then for any $\eta$ in the range
    \begin{equation}\label{eq:eta_range_kappa}
        0 < \eta \leq \of{1 + \frac{2^d\kappa}{n}}^{-1},
    \end{equation}
    the integral over the two-dimensional minor arcs can be bounded by
    \begin{equation*}
        \int_{\nn(\eta)} \abs{S(\bm{\alpha})} \diff \bm{\alpha} \ll X^{n-k-d - \nu}
    \end{equation*}
    for some $\nu > 0$.
\end{lemma}
\begin{proof}
    This is shown using \cite[Lemma 3.4]{brandes_hasse_2021}, with \cref{lemma:minor_mi} in place of \cite[Lemma 3.2]{brandes_hasse_2021}. We have $\rho = 1$, $t = 2^{-d} n$, and $\sigma = 1$, so that
    \begin{equation*}
        \eta_* = \frac{n}{2^d \kappa} \theta_*.
    \end{equation*}
    To complete the proof, we need to make sure that a $\theta_*$ satisfying \cref{eq:theta_star} can indeed be found (and hence a suitable $\eta_*$). Since \cref{eq:range_eta} is equivalent to \cref{eq:eta_range_kappa} under the assumption \cref{eq:kappa_eta_theta} on $\kappa$ and $\eta$, taking
    \begin{align*}
        \eta_{\max} := \of{1 + \frac{2^d\kappa}{n}}^{-1}, && \theta_{\max} := \frac{2^d\kappa}{n + 2^d\kappa},
    \end{align*}
    gives a permissible $\theta_{\max}$, and now
    \begin{equation*}
        \theta_{\max} > \frac{2^d d}{n-2^{d+1}}
    \end{equation*}
    is equivalent to \cref{eq:kappa_n_1}.
\end{proof}

Suppose now that \cref{eq:n20} holds. It is clear that one can find a $\kappa$ that satisfies $n > 2^{d-1} \kappa$ as well as the hypotheses of \cref{lemma:minor_nn}. First, \cref{lemma:minor_mi} ensures that the contribution from $\mi(\theta)$ is suitably small, and then \cref{lemma:subdivision} guarantees that for any $\eta$, either $\bm{\alpha}$ lies on the major arcs $\Na(\eta)$, or $S(\bm{\alpha})$ is small (in terms of $\kappa$ and $\eta$). Next, by \cref{lemma:minor_nn}, for sufficiently small $\eta$, the contribution from the two-dimensional minor arcs $\nn(\eta)$ is acceptable. Hence for any sufficiently small $\eta > 0$, the counting function is
\begin{equation}\label{eq:only_major_left}
    N_{F, G}(X) = \int_{\Na(\eta)} S(\bm{\alpha}) \diff \bm{\alpha} + O \of{X^{n - k-d - \nu}}. 
\end{equation}
It remains only to evaluate the contribution from the major arcs $\Na(\eta)$, which is the objective of the next section.

\section{Major arc contribution}\label{sec:major}
Next, we turn to the analysis of the major arcs. This is performed exactly as in \cite[\S 4]{brandes_hasse_2021}, with the proofs of \cite[Lemmas 4.1 to 4.4]{brandes_hasse_2021} copied to give our \cref{lemma:J,lemma:S}. The changes of notation will be indicated where necessary.

We can begin by enlarging the set of major arcs slightly. 
\begin{definition}
    Define
    \begin{equation}\label{eq:omega}
        \omega = (d-1)\eta + \theta.
    \end{equation}
    With this, let $\Pa(\omega)$ be the set of $\bm{\alpha} \in \T^2$ with the property that there is a natural number $s \leq c ' X^{\omega}$, an $\a = (a_k, a_d) \in \Z^2$ and a $\bm{\gamma} = (\gamma_k, \gamma_d)$ such that
    \begin{equation}\label{eq:alphas_major}
        \alpha_k = \frac{a_k}{s} + \gamma_k \quad \text{and} \quad \alpha_d = \frac{a_d}{s} + \gamma_d,  
    \end{equation}
    with
    \begin{equation*}
        \abs{\gamma_k} = \abs{\alpha_k - \frac{a_k}{s}} \leq c' X^{-k + \omega} \quad \text{and} \quad \abs{\gamma_d} = \abs{\alpha_d - \frac{a_d}{s}} \leq c' X^{-d + \omega}
    \end{equation*}
    for some suitable constant $c'$.
\end{definition}
Two remarks are again in order:
\begin{enumerate}
    \item Looking at \cref{def:najor_arcs}, it is clear that if we set $s = qr$, we can take $c'$ such that $\Na(\eta) \subseteq \Pa(\omega)$. Then assuming \cref{eq:n20}, the result of \cref{lemma:minor_nn} tells us that there is some $\nu > 0$ such that
    \begin{equation*}
        \int_{\Pa(\omega) \setminus \Na(\eta)} S(\bm{\alpha}) \diff \bm{\alpha} \ll \int_{\nn(\eta)} \abs{S(\bm{\alpha})} \diff \bm{\alpha} \ll X^{n - k-d - \nu}.
    \end{equation*}
    \item We can compute directly:
    \begin{equation*}
        \operatorname{vol} \Pa(\omega) \ll \sum_{1 \leq s \leq c' X^{\omega}} \sum_{1 \leq a_k \leq s} c' X^{-k + \omega} \sum_{1 \leq a_d \leq s} c' X^{-d + \omega} \ll X^{-k-d + 5 \omega}.
    \end{equation*}
\end{enumerate}

Let
\begin{equation*}
   S(s; \a) = S(s; a_k, a_d) := \sum_{\x \Mod s} e_s \of{a_k F(\x) + a_d G(\x)},
\end{equation*}
and recall the definition of $J(\bm{\gamma})$ in \cref{eq:J_def}. To make the notation simpler, set
\begin{equation*}
    J_X(\bm{\gamma}) := X^n J(X^k \gamma_k, X^d \gamma_d) = \int_{[-X, X]^n} e \of{\gamma_k F(\bm{\xi}) + \gamma_d G(\bm{\xi})} \diff \bm{\xi}.
\end{equation*}

\begin{lemma}\label{lemma:rewrite_major_int_1}
    The integral over the major arcs can be written as
    \begin{equation*}
        \int_{\Pa(\omega)} S(\bm{\alpha}) \diff \bm{\alpha} = \sum_{s \leq c' X^{\omega}} s^{-n} \sum_{\substack{1 \leq a_k, a_d \leq s \\ (s, a_k, a_d) = 1}} S(s; \a) \int_{\substack{\abs{\gamma_k} \leq c' X^{-k + \omega} \\ \abs{\gamma_d} \leq c' X^{-d + \omega}}} J_{X} (\bm{\gamma}) \diff \bm{\gamma} + O \of{X^{n-k-d-1 + 7 \omega}}.
    \end{equation*}
\end{lemma}
\begin{proof}
    For each $\bm{\alpha} \in \Pa(\omega)$, the separation in \cref{eq:alphas_major} yields the estimate
    \begin{equation}\label{eq:rewrite_S_q_J}
        \abs{S(\alpha_k, \alpha_d) - s^{-n} S(s; \a) J_{X}(\bm{\gamma})} \ll X^{n-1} s \of{1 + X^k \abs{\gamma_k} + X^d \abs{\gamma_d}},
    \end{equation}
    and the right-hand side is $\ll X^{n-1+ 2 \omega}$ for $\bm{\alpha} \in \Pa(\omega)$ by construction.
\end{proof}

\begin{definition}
    With a large number $R$, define the truncated singular series
    \begin{equation*}
        \Series(R) = \sum_{1 \leq s \leq R} s^{-n} \sum_{\substack{a_k, a_d = 1 \\ (s, a_k, a_d) = 1}}^s S(s; \a),
    \end{equation*}
    and the truncated singular integral
    \begin{equation*}
        \Jintegral(R) = \int_{\substack{\abs{\gamma_k} \leq R \\ \abs{\gamma_d} \leq R}} J(\bm{\gamma}) \diff \bm{\gamma}.
    \end{equation*}
\end{definition}

A variable change $\gamma_k' = X^k \gamma_k$ and $\gamma_d' = X^d \gamma_d$ in \cref{lemma:rewrite_major_int_1} shows that
\begin{equation}\label{eq:rewrite_major_int_2}
    \int_{\Pa(\omega)} S(\bm{\alpha}) \diff \bm{\alpha} = X^{n - k-d} \Series(c'X^\omega) \Jintegral(c'X^{\omega}) + O \of{X^{n - k-d-1 + 7 \omega}}.
\end{equation}

We assume now that \cref{eq:n20} holds, and consider each part separately.
\begin{lemma}\label{lemma:J}
    Let $J(\bm{\gamma})$ and $\Jintegral(R)$ be as above.
    \begin{enumerate}[label={(\roman*)}]
        \item For any $\bm{\gamma} =(\gamma_k, \gamma_d) \in \R^2$, we have \label{item:J_bound}
        \begin{equation*}
            \abs{J(\bm{\gamma})} \ll \min \of{1, \abs{\gamma_k}^{-\frac{1}{2^d}n + \eps}, \abs{\gamma_d}^{- \of{\frac{d-1}{\kappa} + \frac{2^d}{n}}^{-1} + \eps}}.
        \end{equation*}
        \item Suppose that \cref{eq:kappa_n_2} holds. Then the limit $\Jintegral := \lim_{R \to \infty} \Jintegral(R)$ exists, and there is a $\nu > 0$ such that \label{item:J_conv}
        \begin{equation*}
            \abs{\Jintegral(2R) - \Jintegral(R)} \ll R^{- \nu}. 
        \end{equation*}
    \end{enumerate}
\end{lemma}
\begin{proof}
    This is:
    \begin{enumerate}[label={(\roman*)}]
        \item \cite[Lemma 4.1]{brandes_hasse_2021} with $\rho = 1$, $(\gamma, \bm{\delta}) = (\gamma_k, \gamma_d)$, $t = 2^{-d} n$ and $\sigma = 1$, and 
        \item the analogue of \cite[Lemma 4.2]{brandes_hasse_2021} using \labelcref{item:J_bound}.
    \end{enumerate}
\end{proof}

\begin{lemma}\label{lemma:S}
    Let $S(s; \a)$ and $\Series(R)$ be as above.
    \begin{enumerate}[label={(\roman*)}]
        \item For any $s \in \N$ and any $\a \in (\Z/s\Z)^2$ with $(s, a_k, a_d) = 1$, we have \label{item:S_bound}
        \begin{equation*}
            s^{-n} \abs{S(s; \a)} \ll s^{\eps} \min \set{\of{\frac{s}{(s, a_k)}}^{-\frac{n}{2^d}}, s^{- \of{\frac{d-1}{\kappa} + \frac{2^d}{n}}^{-1}}}.
        \end{equation*}
        \item Suppose that \cref{eq:kappa_n_2} holds. Then the limit $ \Series := \lim_{R \to \infty} \Series(R)$ exists, and there is a $\nu > 0$ such that \label{item:S_conv}
        \begin{equation*}
            \abs{\Series(2R) - \Series(R)} \ll R^{- \nu}. 
        \end{equation*}
    \end{enumerate}
\end{lemma}
\begin{proof}
    This is:
    \begin{enumerate}[label={(\roman*)}]
        \item \cite[Lemma 4.3]{brandes_hasse_2021} with $\rho = 1$, $(a, \b) = (a_k, a_d)$, $t = 2^{-d} n$ and $\sigma = 1$, and
        \item the analogue of \cite[Lemma 4.4]{brandes_hasse_2021} using \labelcref{item:S_bound}.
    \end{enumerate}
\end{proof}

\begin{proof}[Proof of \cref{thm:main}]
    Assume that \cref{eq:n20} holds. First, the treatment of the minor arcs culminating in \cref{eq:only_major_left} boils the problem down to estimating the integral over $\Na(\eta)$. Next, \cref{eq:rewrite_major_int_2} formulates the integral over the enlarged major arcs in terms of $\Series(c'X^\omega)$ and $\Jintegral(c'X^{\omega})$. Using the convergence guaranteed in \cref{lemma:J,lemma:S}, and the fact that the hypothesis \cref{eq:n20} implies that the variables can be chosen to satisfy \cref{eq:kappa_n_2}, one gets
    \begin{equation*}
        N_{F, G}(X) = X^{n - k-d } \Jintegral \Series + O \of{X^{n - k-d - \nu}}
    \end{equation*}
    for some $\nu > 0$. 

    It remains to identify the constant and discuss its positivity. Since $\Jintegral(R)$ is the integral of $J$ over the box $[-R, R]^2$, we have $\Jintegral = \chi_{\infty}$ by definition. Next, the arguments of \cite[Theorem 2.4]{vaughan_hardy-littlewood_1997} show that $\Series = \prod_p \chi_p$, and that there is a positive number $p_0$ such that
    \begin{equation*}
        \frac{1}{2} < \prod_{p > p_0} \chi_p < \frac{3}{2}.
    \end{equation*}
    Since $F$ and $G$ are homogeneous, if $\x$ is a non-singular solution of the system, then so is $\lambda \x$ for every $\lambda \neq 0$, so that a non-singular solution in $\Q_p^n$ yields one in $\Z_p^n$, and one in $\R^n$ yields one in $[-1, 1]^n$. If the system has a non-singular solution in $\Z_p^n$, then Hensel's lemma shows that $\Gamma(p^h) \gg p^{h(n-2)}$ for all large $h$, whence $\chi_p > 0$. If this holds for every prime $p$, then the above bound for the tail of the product gives $\Series > 0$. Likewise, the arguments of \cite[Lemma 2 and \S 11]{schmidt_simultaneous_1982} show that $\chi_\infty > 0$ if the system has a non-singular solution in $[-1, 1]^n$. For $k = 3$ and $d = 2$, conditions under which the local densities are positive are discussed in more detail in the final paragraph of \cite{brandes_hasse_2021}. Together, this proves \cref{thm:main}.
\end{proof}

\printbibliography

@article{wooley_nested_2019,
	title = {Nested efficient congruencing and relatives of {Vinogradov}'s mean value theorem},
	volume = {118},
	doi = {10.1112/plms.12204},
	number = {4},
	journal = {Proceedings of the London Mathematical Society. Third Series},
	author = {Wooley, T. D.},
	year = {2019},
	mrnumber = {3938716},
	pages = {942--1016},
}

@article{wooley_rational_2015,
	title = {Rational solutions of pairs of diagonal equations, one cubic and one quadratic},
	volume = {110},
	copyright = {http://doi.wiley.com/10.1002/tdm\_license\_1.1},
	doi = {10.1112/plms/pdu054},
	number = {2},
	journal = {Proceedings of the London Mathematical Society},
	author = {Wooley, T. D.},
	year = {2015},
	pages = {325--356},
}

@book{vaughan_hardy-littlewood_1997,
	address = {Cambridge},
	edition = {2},
	series = {Cambridge {Tracts} {In} {Mathematics}},
	title = {The {Hardy}-{Littlewood} method},
	isbn = {978-0-521-57347-4},
	number = {125},
	publisher = {Cambridge Univ. Press},
	author = {Vaughan, R. C.},
	year = {1997},
}

@incollection{schmidt_simultaneous_1982,
	series = {Progress in {Mathematics}},
	title = {Simultaneous rational zeros of quadratic forms},
	volume = {22},
	booktitle = {Séminaire {Delange}-{Pisot}-{Poitou} ({Théorie} des {Nombres}), {Paris} 1980--1981},
	publisher = {Birkhäuser},
	author = {Schmidt, W. M.},
	year = {1982},
	pages = {281--307},
}

@article{browning_forms_2017,
	title = {Forms in many variables and differing degrees},
	volume = {19},
	doi = {10.4171/jems/668},
	number = {2},
	journal = {Journal of the European Mathematical Society},
	author = {Browning, T. D. and Heath-Brown, D. R.},
	year = {2017},
	pages = {357--394},
}

@article{browning_rational_2015,
	title = {Rational points on intersections of cubic and quadric hypersurfaces},
	volume = {14},
	doi = {10.1017/S1474748014000127},
	number = {4},
	journal = {Journal of the Institute of Mathematics of Jussieu},
	author = {Browning, T. D. and Dietmann, R. and Heath-Brown, D. R.},
	year = {2015},
	pages = {703--749},
}

@article{brandes_hasse_2021,
	title = {The {Hasse} principle for diagonal forms restricted to lower-degree hypersurfaces},
	volume = {15},
	doi = {10.2140/ant.2021.15.2289},
	number = {9},
	journal = {Algebra \& Number Theory},
	author = {Brandes, J. and Parsell, S. T.},
	year = {2021},
	note = {Publisher: MSP},
	pages = {2289--2314},
}

\end{document}